\documentclass[11 pt]{article}
\usepackage{amsfonts}          
\usepackage{color}             
\usepackage{makeidx}           
\usepackage{graphicx}          
\usepackage[mathscr]{eucal}    
\usepackage{amssymb,amsmath}   
\usepackage{amsthm}            
\usepackage{xy}                
\usepackage{hyperref}          
\usepackage{dsfont}
\usepackage{subfig}
\usepackage{graphicx}
\usepackage{float}
\usepackage{booktabs}
\usepackage{mathtools}
\usepackage{braket}
\usepackage[a4paper,top=2.5cm,bottom=2.5cm,left=2.5cm,right=2.5cm,%
bindingoffset=5mm]{geometry}

\newtheoremstyle{grassettoecorsivo}
	{}	
	{}	
	{\itshape}	
	{}	
	{\normalfont\bfseries}	
	{.}	
	{ }	
	{}	
\theoremstyle{grassettoecorsivo}
\newtheorem{thm}{Theorem}[section]
\newtheorem{lemma}[thm]{Lemma}

\newtheorem{prop}[thm]{Proposition}
\newtheorem{ex}[thm]{Example}
\newtheoremstyle{grassettoenormale}
	{}	
	{}	
	{\normalfont}	
	{}	
	{\normalfont\bfseries}	
	{.}	
	{ }	
	{}	
\theoremstyle{grassettoenormale}

\newenvironment{sistema}
{\left\lbrace\begin{array}{@{}l@{}}}
{\end{array}\right.}
\DeclarePairedDelimiter{\abs}{\lvert}{\rvert}
\DeclarePairedDelimiter{\norm}{\lVert}{\rVert}
\DeclarePairedDelimiter{\graf}{\lbrace}{\rbrace}

\newcommand{\ack}{\noindent \textbf{Acknowledgements.} \;}

\newcommand{\hypotheses}{\noindent \textsc{Hypotheses} \;}
\newcommand{\remark}{\noindent \textsc{Remark} \;}              
\newcommand{\remarks}{\noindent \textsc{Remarks} \;}              
\newcommand{\opn}[1]{\operatorname{#1}}                         
\newcommand{\crf}[1]{\hyperref[#1]{(\ref{#1};pg.\pageref{#1})}} 
\newcommand{\E}{\mathbb{E}}

\newcommand{\U}{\mathcal{U}}
\newcommand{\R}{\mathbb{R}}

\newcommand{\mP}{\mathbb{P}}

\newcommand{\F}{\mathcal{F}}

\newcommand{\mF}{\mathcal{F}}

\newcommand{\spazio}{\left( \Omega, \F, \mP \right)}

\title{A stochastic maximum principle with dissipativity conditions}
\author{\sc Carlo Orrieri\thanks{Email: carlo.orrieri01@ateneopv.it}\\[4pt]
Dipartimento di Matematica, Universit\`a di Pavia,\\ via Ferrata 1, 27100 Pavia, Italia}
\begin{document}

\maketitle
\begin{abstract}
\noindent In this paper we prove a version of the maximum principle, in the sense of Pontryagin, for the optimal control of a finite dimensional stochastic differential equation, driven by a multidimensional Wiener process. We drop the usual Lipschitz assumption on the drift term and substitute it with dissipativity conditions, allowing polymonial growth. The control enter both the drift and the diffusion term and takes values in a general metric space. 
\end{abstract}

\section{Introduction}
Stochastic maximum principle (SMP for brevity) is a standard tool in order to provide necessary conditions for optimal control problems. After the well known paper by Peng \cite{Peng} for finite dimensional systems, there have been a large number of works on this subject. Firstly, X.Y. Zhou simplified Peng's proof in \cite{Zhou} and studied the relationship between the SMP and dynamic programming. A detailed exposition of this work is contained in \cite{YZ}. Later, a generalization with random coefficients and without $L^p$-bounds on the control is formulated in \cite{CK} for linear equations by Cadenillas and Karatzas. Moreover, other directions have been followed. For example, in \cite{OS,Tang_li_1} the authors studied a version of the SMP for a class of noises with jumps and in \cite{BMO} the case of non-smooth coefficients of the state equation have also been treated. 
Regarding the infinite dimensional case there are still open issues. Indeed most of the results are only concerned with a convex control domain, or with the case in which the diffusion term does not depend on the control \cite{Bensoussan,Hu_Peng}. Recently some works are devoted to the study of a general infinite dimensional SMP. See e.g. \cite{Du_Meng_1,Du_Meng_2,FHT,Lu_Zhang,Tang_li_2}.
   
In this paper we are interested in formulating another version of the general SMP for the optimal control of a stochastic differential equation in a finite dimensional setting, driven by a multidimensional Wiener process. More precisely, we drop the lipschitzianity assumption on the drift term and we replace it with a more natural sign condition, known as dissipativity or monotonicity. This condition is widely studied in the literature both in finite and infinite dimension. In particular, Peng introduced it in \cite{Peng2}, in the study of backward stochastic differential equations (BSDEs in the following) with random terminal time. Then let us mention the contributions of Pardoux \cite{Pardoux}, Briand et al. \cite{Briand}, and more recently Briand and Confortola \cite{Briand_Conf}.
If the equation of the state is
\begin{equation*}
\begin{sistema}
dx(t) = b(t,x(t),u(t))dt + \sigma(t,x(t),u(t))dW(t) \\
x(0) = x_0
\end{sistema}
\end{equation*}
with cost functional of the form
\begin{equation*}
J(u(\cdot)) = \E \left[ \int_0^T f(t,x(t),u(t))dt + h(x(T))\right]
\end{equation*}
then the dissipativity of $b$ is expressed by the following
\begin{equation}\label{DISSIPATIVITY}
\braket{b(t,x,u) - b(t,x',u), x-x'} \leq \alpha\abs{x-x'}^2, \qquad \mP\text{-a.s. } \; t\in [0,T],\; u \in U;
\end{equation}
for every $x,x' \in \R^n$ and some constant $\alpha \in \R$. The key fact is that the dissipativity condition is inherited by both the first and second variation equation of the state. Indeed a condition of the following type holds $\mP$-a.s., for all $y$ in $\R^n$, for the drift term
\begin{equation}
\braket{D_xb(t,x,u)y,y} \leq \alpha\abs{y}^2, \qquad t \in [0,T], \; u\in U.
\end{equation}
Moreover, also the BSDE arising as first adjoint equation satisfies such a condition and this fact enables us to prove well posedness via an existence and uniqueness result provided by P. Briand et al. in \cite{Briand}. For the second adjoint equation, which is matrix valued, we get the same result equipping the space with the Hilbert-Schmidt norm. \\ 
Regarding the diffusion term of the state equation, we impose the usual Lipschitz hypothesis and we assume that both coefficients are regular in $x$. For the cost functional, we allow a polynomial growth for the coefficients. Finally, we stress the fact that, with these assumptions, we are able to treat polynomials of odd degree with strictly negative leading term as drift coefficients, instead of imposing only a linear growth. 

The paper is organized as follows. In Section 2 we fix notations and assumptions, we introduce the adjoint BSDEs and we state some preliminary results on the stochastic differential equation of the state. In Section 3 we give the statement of the main result. Section 4 is devoted to the spike variation technique and to the expansion of the cost. In Section 5 we conclude the proof of the SMP. Finally, Section 6 is devoted to the case for which the domain of the controls is convex and only one adjoint equation is needed. A sufficient condition of optimality is also exhibited. 

\section{Notations and Preliminaries}
Throughout this paper we let $W = \graf{W^1(t), \ldots,W^d(t)}_{t\geq 0}$ be a standard $d$-dimensional Brownian motion defined on some complete probability space $\spazio$. We denote by $(\mF_t)_{t\geq 0}$ the natural filtration associated to $W$, satisfying the usual conditions. We suppose that all the processes are defined for times $t \in [0,T]$. Then we denote by $\mathcal{P}$ the $\sigma$-algebra on $\Omega \times [0,T]$ generated by progressive processes. For any $p \geq 1$ we define 
\begin{itemize}
\item $L^p_{\mF}(0,T;\R^n)$: the set of all $\graf{\mF_t}_{t\geq 0}$-progressive processes $x(\cdot)$ such that \\$\E\int_0^T\abs{x(t)}^p dt <\infty$.
\item $C([0,T];L^p(\Omega;\R^n))$: the set of all $\graf{\mF_t}_{t\geq 0}$-progressive processes $x(\cdot)$ such that the map $t \mapsto x(t) \ni L^p(\Omega)$ is continuous and $\E \sup_{t \in [0,T]}\abs{x(t)}^p <\infty$.
\end{itemize}
The space of control actions is a general metric space $U$ (except in section 6), endowed with its Borel $\sigma$-algebra $\mathcal{B}(U)$. Furthermore, the class of admissible controls is defined by requiring that they are progressively measurable with respect to $\graf{\mF_t}_{t \geq 0}$, more precisely
\[ \U[0,T]:= \graf{u(\cdot): [0,T]\times \Omega \rightarrow U : u(\cdot) \text{ is } (\mF_t)_{t\geq 0}-\text{progressive}}. \]
Finally we will denote by $\abs{\cdot}$ the Euclidean norm in $\R^n$.\\ 

Here we want to study a finite horizon stochastic control problem in the form 
\begin{equation}\label{SDE}
\begin{sistema}
dx(t) = b(t,x(t),u(t))dt + \sigma(t,x(t),u(t))dW(t) \qquad t \in [0,T]\\
x(0) = x_0,
\end{sistema}
\end{equation}
with a cost functional given by
\begin{equation}\label{COST}
J(u(\cdot)) = \E \left[ \int_0^T f(t,x(t),u(t))dt + h(x(T))\right],
\end{equation}
If $x(\cdot)$ is a solution of \eqref{SDE} and $u(\cdot) \in \U[0,T]$ then we call $(x(\cdot),u(\cdot))$ an \textit{admissible pair}.
The control problem can be formulated as a minimization of the cost over $\U[0,T]$, more precisely a control $\bar{u}$ is \textit{optimal} if 
\begin{equation}\label{PROBLEMA}
J(\bar{u}) = \inf_{u(\cdot) \in \U[0,T]} J(u(\cdot)).
\end{equation}

\hypotheses 
\begin{itemize}
\item[(H1)] $(U,d)$ is a separable metric space;
\item[(H2)] The drift term $b:\Omega \times [0,T] \times \R^n \times U \rightarrow \R^n$ is $\mathcal{P}\otimes \mathcal{B}(\R^n) \otimes \mathcal{B}(U)$-measurable, where $\mathcal{P}$ is the progressive $\sigma$-algebra. The map $x \mapsto b(t,x,u)$ is $C^2(\R^n;\R^{n})$ and satisfies a 
\textit{$\alpha$-dissipativity} condition in the sense that there exists a constant $\alpha \in \R$ such that, $\mP$-a.s.
\begin{equation}\label{DISSIPATIVITY}
\braket{b(t,x,u) - b(t,x',u), x-x'} \leq \alpha\abs{x-x'}^2, \qquad u \in U, \; t\in [0,T],\; x,x' \in \R^n.
\end{equation}
\item[(H3)] The diffusion coefficient $\sigma: \Omega \times [0,T] \times \R^n \times U \rightarrow \R^{n\times d}$ is measurable with respect to $\mathcal{P}\otimes \mathcal{B}(\R^n) \otimes \mathcal{B}(U)$. Moreover the map $x \mapsto \sigma(t,x,u)$ is $C^2(\R^n;\R^{n\times d})$ and there exists a constant $C_1 >0$ such that, $\mP$-a.s.
\[ \abs{\sigma(t,x,u) - \sigma(t,x',u)} \leq C_1\abs{x-x'}, \qquad u \in U,\; t \in [0,T],\; x,x' \in \R^n.\]
\item[(H4)] (\textit{Polynomial Growth}) There exist $h \geq 0$, $C_2 > 0$ such that, for $j = 0,1,2$, $\mP$-a.s. 
\[\sup_{u \in U}\sup_{t \in [0,T]}\abs{D^{\beta}_xb(t,x,u)} \leq C_2(1 + \abs{x}^{h}),\qquad \abs{\beta} = j.\]
In addition we shall assume there exist $C_3 >0$, $C_4 >0$, $k \geq 0$ such that, $\mP$-a.s.
\[ \abs{\sigma(t,0,u)} \leq C_3, \qquad \qquad \qquad \qquad u \in U,\; t \in [0,T];\] 
\[ \sup_{u \in U}\sup_{0 \leq t\leq T} \abs{D^2_x\sigma(t,x,u)} \leq C_4(1+ \abs{x}^{k}), \qquad x \in \R^n.\]
\item[(H5)] $f:[0,T]\times \R^n \times U \rightarrow \R$ and $h: \R^n \rightarrow \R$ are measurable and the maps $x \mapsto f(t,x,u)$ and $x \mapsto h(x)$ are $C^2(\R^n;\R)$. Moreover there exists $C_5 > 0$, $m\geq 0$, $l \geq 0$ such that for $j = 0,1,2$ we have, $\mP$-a.s.
\[\sup_{u \in U}\sup_{t \in [0,T]} \abs{D^{\beta}_x f(t,x,u)} \leq C_5(1 + \abs{x}^{l}),\qquad \;\abs{\beta} = j,\]
\[\abs{D^{\beta}_x h(x)} \leq C_5(1 + \abs{x}^{m}),\qquad \qquad \qquad \qquad \abs{\beta} = j.\]
Moreover we suppose that for $\varphi = b,\sigma,f$, $u \mapsto \varphi(t,x,u)$ is continuous for every $x \in \R^n, t \in [0,T]$
\end{itemize}

\remarks 
\begin{enumerate}
\item Hypothesis (H3)-(H4) implies in particular that $\sigma(t,\cdot,u)$ has linear growth and there exists a constant $C_6$ independent of $\omega, t$ and $u$ such that $\abs{D_x b(t,x,u)} \leq C_6$ is bounded. Thanks to the polynomial growth of all the maps involved $\varphi = b,\sigma,f,D_xb,D_x\sigma,D_xf$, $D^2_xb,D^2_x\sigma,D^2_xf$ and hypothesis (H4) we have also that \[\abs{\varphi(t,0,u)} \leq L \] 
for some positive constant $L>0$.
\item The $\alpha$-dissipativity in (H2) can be rewritten, $\mP$-a.s.
\[ \braket{(b(t,x,u) - \alpha x) - (b(t,x',u)- \alpha x') , x-x'} \leq 0 \qquad t \in [0,T], \; u\in U, \; x,x' \in \R^n\]
which states that the function $b(t,x,u) - \alpha x$ is dissipative. 
\item It is easy to see that also the derivative of $b(t,x,u)$ with respect to $x$ satisfies a dissipativity condition. Indeed, for all $y \in \R^n$, $\mP$-a.s.
\begin{equation}\label{eq.diss.var}
\braket{D_xb(t,x,u)y,y} \leq \alpha\abs{y}^2 \qquad t \in [0,T], \; u\in U.
\end{equation}
This property is crucial in order to guarantee the well-posedness of the first and second variation of the state equation. 
\end{enumerate}

\noindent The following result is essential for the well posedness of the optimal control problem and it concerns the existence and uniqueness of a solution to the state equation, for any $u \in \U[0,T]$. Although it is known in the literature (cfr. e.g. \cite{DpIT}), we provide a sketch of the proof for completeness.

\begin{prop}\label{p.ex_apriori}
Under assumptions (H1)-(H5) the stochastic equation \eqref{SDE} admits a unique solution in $C([0,T];L^2(\Omega;\R^n))$, i.e. a progressive process $x(t)$ satisfying 
\[ \sup_{t \in[0,T]} \E\abs{x(t)}^2 <\infty.\]
Moreover, there exists a constant $C = C(T,p)$ dependent on $T$ and $p \geq1$, such that
\begin{equation}\label{Lp_estimate}
\sup_{t \in [0,T]}\E\abs{x(t)}^p \leq C(1+\abs{x_0}^p), \qquad \qquad p\geq 1.
\end{equation}
\end{prop}

\begin{proof}
To simplify the notation we drop the dependence on the control; the case of controlled equation can be treated exactly in the same way.
By fixing $\gamma \in C([0,T];L^2(\Omega;\R^n))$ we want to show that the problem
\[ dx(t) = b(t,x(t))dt + \sigma(t,\gamma(t))dW(t), \qquad x(0) =x_0\]
admits a unique solution $J(\gamma)$ which belongs to $C([0,T];L^2(\Omega;\R^n))$. The existence part follows from the fact that the initial problem can easily reformulated as a differential equation with random coefficients of the form 
\begin{equation}
\dfrac{d}{dt}\eta(t) = b(t,\eta(t) + w^{\gamma}(t))
\end{equation}
where the quantity
\[ w^{\gamma}(t) := \int_0^t \sigma(s,\gamma(s))dW(s) \]
is well defined thanks to the linear growth imposed by the Lipschitz assumption. 
Since $b(\cdot)$ is continuous, we know that there is a local solution which can be easily extended to the whole $[0,T]$, by the dissipativity assumptions. Now we have to verify that the operator $J:C([0,T_0];L^2(\Omega;\R^n)) \rightarrow C([0,T_0];L^2(\Omega;\R^n))$ is a contraction if $T_0$ is small enough. Applying It\^{o}'s formula and taking expectation we get, for any $\gamma_1$, $\gamma_2 \in C([0,T];L^2(\Omega;\R^n))$
\begin{equation*}
\begin{split}
\E \abs{J_t(\gamma_1) - J_t(\gamma_2)}^2 &= 2\E \int_0^t \braket{ b(s,J_s(\gamma_1)) - b(s,J_s(\gamma_2)), J_s(\gamma_1) - J_s(\gamma_2)}ds \\
&+ \E \int_0^t \norm{\sigma(s,\gamma_1(s)) - \sigma(s,\gamma_2(s))}_2^2ds \\
&\leq 2\alpha \int_0^t \E \abs{J_s(\gamma_1) - J_s(\gamma_2)}^2ds + C_1^2 \int_0^t \E \abs{\gamma_1(s) - \gamma_2(s)}^2 ds,
\end{split}
\end{equation*} 
and so
\begin{equation}
\E \abs{J_t(\gamma_1) - J_t(\gamma_2)}^2 \leq C_1^2\int_0^t e^{2\alpha(t-s)}\E\abs{\gamma_1(s) - \gamma_2(s)}^2 ds;
\end{equation}
where we used assumptions on coefficients and the Gronwall lemma.
Eventually, 
\[ \sup_{t \in [0,T]}\E \abs{J_t(\gamma_1) - J_t(\gamma_2)}^2 \leq C_1^2e^{2\alpha T}T \sup_{t \in [0,T]}\E\abs{\gamma_1(t) - \gamma_2(t)}^2 \]
and if we choose $T_0$ such that $C_1\sqrt{T_0}e^{\alpha T_0} <1$ we prove that $J$ is a contraction. Proceeding in the same way on $[T_0,T_1],[T_1,T_2],\ldots$ we find a unique solution defined on the whole $[0,T]$. 
The estimate \eqref{Lp_estimate} of a generic momentum of the solution follows easily applying Ito's formula.
\end{proof}

\noindent Now we have to deal with the two backward stochastic differential equations arising as adjoint equations with terminal conditions in the formulation of the SMP. 
The first order adjoint equation has the following form 

\begin{equation}\label{eq.adjoint_first}
\begin{cases}
dp(t) = - \left[ D_xb(t,\bar{x}(t),\bar{u}(t))^T p(t)  + \sum_{j=1}^dD_x\sigma^j(t,\bar{x}(t),\bar{u}(t))^T q_j(t)-D_xf(t,\bar{x}(t),\bar{u}(t)) \right]dt\\
\qquad \qquad + \sum_{j=1}^d q_j(t) dW^j(t)\\
p(T) = -D_x h(\bar{x}(T));
\end{cases}
\end{equation}
where $p(\cdot)$ is the first order adjoint process and $\bar{x}(t)$, $\bar{u}(t)$ are the optimal trajectory and the optimal control process, respectively. It is even worth noting that the coefficient $D_xb(t,\bar{x}(t),\bar{u}(t))$ in front of $p(t)$ is dissipative, as we observed in Remark 3; this is the key fact in order to check the well posedness of the equation. Then, as it was pointed out in \cite{Peng}, the presence of the control in the diffusion term forces the introduction of a second variation process, which can be represented as the solution of a matrix valued BSDE of the form 

\begin{equation}\label{eq.adjoint_second}
\begin{cases}
dP(t) = - \left[ D_xb(t,\bar{x}(t),\bar{u}(t))^T P(t)  + P(t)D_xb(t,\bar{x}(t),\bar{u}(t))\right]dt \\
\qquad \qquad + \sum_{j=1}^d D_x\sigma^j(t,\bar{x}(t),\bar{u}(t))^T P_j(t)D_x\sigma^j(t,\bar{x}(t),\bar{u}(t))dt\\
\qquad \qquad + \sum_{j=1}^d \left( D_x\sigma^j(t,\bar{x}(t),\bar{u}(t))^T Q_j(t)+ Q_j(t)D_x\sigma^j(t,\bar{x}(t),\bar{u}(t)) \right)dt\\
\qquad \qquad + D_x^2 H(t,\bar{x}(t),\bar{u}(t),p(t),q(t))dt + \sum_{j=1}^d Q_j(t) dW^j(t)\\
p(T) = -D_x^2 h(\bar{x}(T));
\end{cases}
\end{equation}

where $H$ is the Hamiltonian and it is defined by

\begin{equation}\label{hamiltonian}
H(t,x,u,p,q) = \braket{p,b(t,x,u)} + \opn{Tr}[q^T\sigma(t,x,u)] - f(t,x,u).
\end{equation}
Also in this case we have a kind of monotonicity in the first term.
Now we observe that a solution of the first (second) adjoint BSDE is a pair of adapted processes $(p(\cdot),q(\cdot)) \in L^2_{\mF}(0,T;\R^n) \times (L^2_{\mF}(0,T;\R^n))^d$ (respectively, $ (P(\cdot),Q(\cdot)) \in L^2_{\mF}(0,T;S^n) \times (L^2_{\mF}(0,T;S^n))^d$), where $S^n$ is the space of symmetric matrices. Indeed, the following theorem hold
\begin{thm}\label{t.BSDE}
Under hypotheses (H1)-(H5) the adjoint equation \eqref{eq.adjoint_first} has a unique adapted solution $ (p(\cdot),q(\cdot)) \in L^2_{\mF}(0,T;\R^n) \times (L^2_{\mF}(0,T;\R^n))^d$.
\end{thm}  
\begin{proof}
Thanks to the growth assumptions in hypothesis (H5) and \eqref{Lp_estimate} we have that 
\[ \E\left[ \abs{D_x h(\bar{x}(T))}^2 + \left( \int_0^t \abs{D_xf(t,\bar{x}(t),\bar{u}(t))} \right)^2 \right] <\infty \]
and $\forall r >0$
\[ \sup_{\abs{p}\leq r} \abs{D_xb(t,\bar{x}(t),\bar{u}(t))^T p(t)} \in L^1_{\mathcal{F}}(0,T;\R^n)\]
Hence, using the result in \cite{Briand}, Theorem 4.1, page 119, the hypotheses of the theorem are satisfied by the BSDE \eqref{eq.adjoint_first} and we have finished. 
\end{proof}

\noindent Regarding the second adjoint equation we need to check that the drift term remains dissipative even though the BSDE is matrix valued. To do it we introduce the Hilbert-Schmidt norm and the corresponding scalar product in $S^n$ as
\[\braket{A,B}_2 = \opn{Tr}(AB^T) \qquad \forall A,B \in S^n.\]
Then we can state the following
\begin{thm}
Under hypotheses (H1)-(H5) the adjoint equation \eqref{eq.adjoint_second} has a unique adapted solution $(P(\cdot),Q(\cdot)) \in L^2_{\mF}(0,T;S^n) \times (L^2_{\mF}(0,T;S^n))^d$.
\end{thm}
\begin{proof}
Let us consider firstly the scalar product in the Hilbert-Schmidt norm 
\[ \braket{D_x b(t) P, P}_2 = \opn{Tr}(D_x b(t) PP^T) \]
where $P \in S^n$ and $D_xb(t)=D_xb(t,\bar{x}(t),\bar{u}(t))$.
Now decompose $PP^T$ in the following way
\[ PP^T = \sum_{i=1}^n \gamma_i c_i \cdot c_i^T \]
where $\gamma_i \geq 0$ and $c_i$ are the eigenvalues and the (orthonormal) eigenvectors of $PP^T$, respectively. Then we have
\begin{equation}\label{dissipativity_matrix}
\begin{split}
\braket{D_x b(t) P, P}_2 &= \sum_{i=1}^n \gamma_i \opn{Tr}(D_xb(t)c_i c_i^T) = \sum_{i=1}^n \gamma_i \opn{Tr}(c_i^T D_xb(t) c_i) \leq \alpha \sum_{i=1}^n \gamma_i \abs{c_i}_2^2\\
&= \alpha \sum_{i=1}^n \gamma_i = \alpha \opn{Tr}(PP^T) = \alpha\norm{P}^2_2.
\end{split}
\end{equation}
That is exactly the dissipativity condition we need. In fact, as in Theorem \ref{t.BSDE}, using again the result in \cite{Briand}, Theorem 4.1, and taking into account the dissipativity obtained in \eqref{dissipativity_matrix} we get the required result.
\end{proof}

\section{Statement of the Theorem}
Now we are in position to state the Pontryagin-type stochastic maximum principle for the optimal control problem \eqref{PROBLEMA} associated to the state equation \eqref{SDE}.

\begin{thm}\label{t.SMP}
(SMP)  Suppose (H1)-(H5) hold and let $(\bar{x},\bar{u})$ be an optimal pair for the control problem \eqref{PROBLEMA}. Then there exist two pairs of processes \[
\begin{cases}
(p(\cdot), q(\cdot)) \in L^2_{\mathcal{F}}([0,T],\R^n) \times (L^2_{\mathcal{F}}([0,T],\R^n))^d \\
(P(\cdot), Q(\cdot)) \in L^2_{\mathcal{F}}([0,T],S^n) \times (L^2_{\mathcal{F}}([0,T],S^n))^d
\end{cases}\]
that are solutions to the BSDEs \eqref{eq.adjoint_first} and \eqref{eq.adjoint_second} respectively, such that
\[ \mathcal{H}(t,\bar{x}(t), \bar{u}(t)) = \max_{u \in U}\mathcal{H}(t,\bar{x}(t), u) \qquad \qquad d\mP \times dt \; a.s.\]
where
\begin{equation}
\begin{split}
\mathcal{H}(t,x,u) &:= H(t,x,u,p(t),q(t)) - \dfrac{1}{2}\opn{Tr}\bigl( \sigma(t,\bar{x}(t),\bar{u}(t))^TP(t)\sigma(t,\bar{x}(t),\bar{u}(t)) \bigr)\\
&+ \dfrac{1}{2}\opn{Tr}\bigl[ (\sigma(t,x,u) - \sigma(t,\bar{x}(t),\bar{u}(t))^T)P(t)(\sigma(t,x,u) - \sigma(t,\bar{x}(t),\bar{u}(t)) \bigr]
\end{split}
\end{equation}
\end{thm} 
\begin{ex}\label{ex1} 
{\rm Here we want to stress that, under our assumption, we can consider a state equation which has a drift coefficient of the following type 
\begin{equation}
b_i(t,x) = -c_i(u) x(t)^{2m+1} + \sum_{j=1}^{2m} c_{ij}(u)x(t)^j, \qquad \quad 1\leq i\leq n,
\end{equation} 
where  $0 < \lambda \leq c_i(\cdot) \leq C$ and $\abs{c_{ij}(\cdot)} \leq C$.
i.e. polynomials of degree $2m+1$ with strictly negative leading term. This is a genuine generalization of the classical Lipschitz case in which only a linear growth is allowed.}
\end{ex}
\section{Spike Variation Technique}
In this section we are going to study a Taylor expansion of the state trajectory with respect to a needle perturbation of the control. Let $E_{\varepsilon} \subset [0,T]$ be a set of measure $\varepsilon$ and $\bar{u}$ an optimal control, then we define the perturbed control as
\[
u^{\varepsilon}(t)=
\begin{cases}
\bar{u}(t), & \quad \text{if } t \in [0,T]\setminus E_{\varepsilon} \\
w, & \quad \text{if } t \in E_{\varepsilon},
\end{cases}
\]
\remark In general $U$ does not have a linear structure, hence a perturbation like $\bar{u}(t) + \varepsilon u(t)$ is meaningless unless $U$ is, for example, a convex space. We will discuss this case later.\\
\newline
If $(\bar{x}(\cdot),\bar{u}(\cdot))$ is a given optimal pair, let $(x^{\varepsilon}(\cdot), u^{\varepsilon}(\cdot))$ satisfy the following
\begin{equation}\label{SDE_perturbed}
\begin{sistema}
dx^{\varepsilon}(t) = b(t,x^{\varepsilon}(t),u^{\varepsilon}(t))dt + \sigma(t,x^{\varepsilon}(t),u^{\varepsilon}(t))dW(t) \\
x^{\varepsilon}(0) = x_0.
\end{sistema}
\end{equation}
Following the notation of Yong and Zhou \cite{YZ}, we will denote by $\delta\varphi(t)$ the quantity $\varphi(t,\bar{x}(t),$ $u^{\varepsilon}(t)) - \varphi(t,\bar{x}(t),\bar{u}(t))$, for a generic function $\varphi$, and by $y^{\varepsilon}(\cdot), z^{\varepsilon}(\cdot)$ the solutions of the following SDEs
\begin{equation}\label{SDE_FirstVariation}
\begin{sistema}
dy^{\varepsilon}(t) = D_xb(t)y^{\varepsilon}(t)dt + \sum_{j=1}^{d}\bigl[ D_x\sigma^j(t)y^{\varepsilon}(t) + \delta\sigma^j\chi_{E_{\varepsilon}}(t)\bigr]dW^j(t) \\
y^{\varepsilon}(0) = 0,
\end{sistema}
\end{equation}
and
\begin{equation}\label{SDE_SecondVariation}
\begin{sistema}
dz^{\varepsilon}(t) = \bigl[ D_xb(t)z^{\varepsilon}(t) + \delta b(t)\chi_{E_{\varepsilon}}(t) + \dfrac{1}{2}D^2_x b(t)y^{\varepsilon}(t)^2 \bigr]dt\\
\qquad + \sum_{j=1}^{d}\bigl[ D_x\sigma^j(t)z^{\varepsilon}(t) + \delta D_x\sigma^j(t)y^{\varepsilon}(t)\chi_{E_{\varepsilon}}(t) + \dfrac{1}{2}D^2_x \sigma^j(t)y^{\varepsilon}(t)^2\bigr]dW^j(t) \\
z^{\varepsilon}(0) = 0,
\end{sistema}
\end{equation}
where 
\[D_xb(t):= D_xb(t,\bar{x}(t),\bar{u}(t)), \qquad D_x\sigma^j(t):= D_x\sigma^j(t,\bar{x}(t),\bar{u}(t)),\]
have values in $\R^{n\times n}$ for $1 \leq j \leq d$ and also
$$
{D^2_x b(t)y^{\varepsilon}(t)^2} := \left(
\begin{array}{ccc}
\opn{Tr}\bigl[ D^2_x b^1(t)y^{\varepsilon}(t)y^{\varepsilon}(t)^T \bigr] \\
\vdots \\
\opn{Tr}\bigl[ D^2_x b^n(t)y^{\varepsilon}(t)y^{\varepsilon}(t)^T \bigr]
\end{array}
\right),
$$

$$
{D^2_x \sigma^j(t)y^{\varepsilon}(t)^2} := \left(
\begin{array}{ccc}
\opn{Tr}\bigl[ D^2_x \sigma^{1j}(t)y^{\varepsilon}(t)y^{\varepsilon}(t)^T \bigr] \\
\vdots \\
\opn{Tr}\bigl[ D^2_x \sigma^{nj}(t)y^{\varepsilon}(t)y^{\varepsilon}(t)^T \bigr]
\end{array}
\right).
$$
Here we want to obtain an a priori estimate for a general linear SDE with stochastic coefficients in the spirit of lemma 4.2 of \cite{YZ}, which will be useful in the sequel.

\begin{lemma}\label{l.eq.lin}
Let $Y(t) \in L_{\mathcal{F}}^{2}(0,T;\R^n)$ be a solution of the following linear SDE
\[\begin{cases}
dY(t) = \left( A(t)Y(t) + \alpha(t) \right)dt + \sum_{j=1}^d\left( B^j(t)Y(t) + \beta^j(t) \right)dW^j(t)\\
Y(0) = Y_0;
\end{cases}\]
where $A,B^j: [0,T]\times \Omega \rightarrow \R^{n\times d}$ and $\alpha,\beta^{j}:[0,T]\times \Omega \rightarrow \R^{n}$ are $(\F_t)$-progressive. Moreover, suppose that the following conditions hold:
there exist $c \in \R$, $L \geq 0$, $k\geq 1$ such that
\begin{enumerate}
\item $\braket{A(t)Y(t),Y(t)} \leq  c\abs{Y(t)}^2 \;\qquad \qquad \qquad \qquad \qquad d\mP\times dt-\text{a.s.}$;
\item $\abs{B^j(t)} \leq L, \;\;\qquad \qquad \qquad \qquad \qquad \qquad \qquad \qquad d\mP\times dt-\text{a.s.},\, 1\leq j\leq d$;
\item $\int_0^T \graf{\E \abs{\alpha(s)}^{2k}}^{\frac{1}{2k}}ds + \int_0^T \graf{\E \abs{\beta(s)}^{2k}}^{\frac{1}{k}}ds < \infty$,$ \qquad $ $1\leq j\leq d$, for some $k\geq 1$.
\end{enumerate}
Then the following a priori estimate holds
\[ \sup_{0 \leq t \leq T} \E\abs{Y(t)}^{2k} \leq K\left[\E\abs{Y_0}^{2k} + \left( \int_0^T \graf{\E\abs{\alpha(s)}^{2k}}^{\frac{1}{2k}}\right)^{2k} + \sum_{j=1}^d \left( \int_0^T \graf{\E\abs{\beta^j(s)}^{2k}}^{\frac{1}{2k}}ds \right)^k \right]. \]
\end{lemma}

\begin{proof}
We exhibit the proof for $\alpha$ and $\beta$ being bounded; the general case follows using the usual approximation argument. We begin by computing It\^{o}'s formula for $f(y) = \abs{y}^p$, with $p \geq 4$, and we set $p = 2k$. The case with $p\in [1,4)$ follows from the H\"{o}lder inequality. 
\begin{equation*}
\begin{split}
\E\abs{Y(t)}^{2k} &\leq \E\abs{Y_0}^{2k} + 2k\E \int_0^t \abs{Y(s)}^{2k-2}\braket{Y(s),A(s)Y(s) + \alpha(s)}ds\\
&+ 2k(k-1)\sum_{j=1}^d\E \int_0^t \abs{Y(s)}^{2k-4}\braket{Y(s),B_j(s)Y(s) + \beta_j(s)}^2 ds \\
&+ \sum_{j=1}^d k\E\int_0^t\abs{Y(s)}^{2k-2}\braket{B_j(s)Y(s) + \beta_j(s),B_j(s)Y(s) + \beta_j(s)} ds \\
&\leq \E\abs{Y_0}^{2k} + K \E\int_0^t \left[ \abs{Y(s)}^{2k} \right.\\
&+ \left. \abs{Y(s)}^{2k-1}(\abs{\alpha(s)} + \abs{\beta(s)}) + \abs{Y(s)}^{2k-2}\abs{\beta(s)}^2 \right]ds.
\end{split}
\end{equation*}
Then using Young inequality twice, we obtain
\[ \E\abs{Y(t)}^{2k} \leq \E\abs{Y_0}^{2k} + K\E\int_0^t\left[ \abs{Y(s)}^{2k} + \abs{\alpha(s)}^{2k} + \abs{\beta(s)}^{2k} \right]ds \]
Hence, from Gronwall inequality we get
\[ \E\abs{Y(t)}^{2k} \leq K\left( \E\abs{Y_0}^{2k} + \E\int_0^T \left[ \abs{\alpha(s)}^{2k} + \abs{\beta(s)}^{2k} \right]ds \right). \] 
Now, in order to obtain the required estimate under hypothesis (3) we set
\[ \varphi(t) = \left( \sup_{0\leq s\leq t} \E\abs{Y(t)}^{2k}\right)^{1/2k} \]
and following \cite{YZ}, we end up with
\[\varphi(T)^{2k} \leq K \biggr\{\varphi(0)^{2k} + \left[ \int_0^T (\E\abs{\alpha(s)}^{2k})^{1/2k}ds \right]^{2k} + \left[\int_0^T (\E\abs{\beta(s)}^{2k})^{1/k}ds \right]^{k}\biggr\}.
\]
which is the required result.
\end{proof}

Concerning the well posedeness of the stochastic differential equations \eqref{SDE_FirstVariation} and \eqref{SDE_SecondVariation}, the key fact is that on the drift term of both the equations we have a dissipativity condition like \eqref{eq.diss.var}. This enables us to state the following 

\begin{prop}
Under the hypotheses (H1)-(H5) the stochastic differential equations \eqref{SDE_FirstVariation} and \eqref{SDE_SecondVariation} admit a unique solution $y^{\varepsilon}, z^{\varepsilon} \in C([0,T];L^2(\Omega;\R^n))$.
\end{prop}

\begin{proof}
The proof is standard and it will be only sketched.
Let us begin with the first variation equation. As in the proof of Proposition \ref{p.ex_apriori} we consider for simplicity the case where the SDE is not controlled, then we reduce to an equation of the form
\begin{equation}
\begin{sistema}
dy(t) = D_xb(t)y(t) + \sum_{j=1}^{d} D_x\sigma^j(t)y(t)dW^j(t) \\
y^{\varepsilon}(0) = 0,
\end{sistema}
\end{equation}
For the controlled one the only difference is that it remains to check that the term $\delta\sigma\chi_{E_{\varepsilon}}$ is integrable, but this is obvious thanks to the growth condition on $\sigma$. Now, we fix $\gamma \in C([0,T],L^2(\Omega,\R^n))$ and we define 
\begin{equation*}
J_t(\gamma) = \int_0^t D_xb(s,x(s))\gamma(s)ds + \int_0^t D_x\sigma(s,x(s))\gamma(s)dW(s), \qquad t\in [0,T].
\end{equation*}
Since $\E \abs{D_xb(s,x(s))} \leq C_1(1 + \E\abs{x(s)}^p) \leq C(1 + \abs{x_0}^p)$, for some $p \geq 1$, and using the boundedness of $D\sigma(t)$, it easy to show that $J_t$ is a contraction in $C([0,T_0],L^2(\Omega,\R^n))$ with $T_0$ small enough. The existence and uniqueness of the solution follow straightforward repeating the argument in the successive intervals $[T_0,T_1],[T_1,T_2],\ldots$ up to $T$.

Regarding the second variation equation we use exactly the same technique. Indeed, for any $\gamma \in C([0,T],L^2(\Omega,\R^n))$ we define
\begin{equation*}
\begin{split}
J_t(\gamma) &= \int_0^t D_xb(s,x(s))\gamma(s)ds + \dfrac{1}{2}\int_o^t D^2_xb(s,x(s))y(s)^2 ds\\
&+ \sum_{j=1}^d \int_0^t D_x\sigma^j(s,x(s))\gamma(s)dW^j(s) + \dfrac{1}{2}\sum_{j=1}^d\int_0^t D^2_x\sigma^j(s,x(s))y(s)^2dW^j(s),
\end{split}
\end{equation*} 
using the boundedness of $D_x\sigma(t)$, the dissipativity of $D_xb(t)$ and the a priori estimate given in Lemma \ref{l.eq.lin} we have again that the map $J_t$ is a contraction in $C([0,T_0],L^2(\Omega,\R^n))$, hence with the same arguments as before we have existence and uniqueness of a solution for the second variation equation. 
\end{proof}

Before exhibiting an expansion of the state with respect to small perturbations of the control process we provide a Taylor formula in the form of a lemma.

\begin{lemma}\label{l.expansion}
If $g \in C^2(\R^n)$ then the following equality holds for every $x, \bar{x} \in \R^n$
\[ g(x) = g(\bar{x}) + \braket{D_xg(\bar{x}), x - \bar{x}} + \int^1_0 \braket{\theta D_x^2 g(\theta\bar{x} + (1-\theta)(x- \bar{x})),x-\bar{x}}d\theta.\] 
\end{lemma}

\noindent The central result of this section is the following

\begin{prop}\label{p.expansion}
Suppose hypotheses (H1)-(H5) hold and define $\xi^{\varepsilon}(t):= x^{\varepsilon}(t) - \bar{x}(t)$, $\eta^{\varepsilon}(t):= \xi^{\varepsilon}(t) - y^{\varepsilon}(t)$ and $\zeta^{\varepsilon}(t):=\xi^{\varepsilon}(t) - y^{\varepsilon}(t) - z^{\varepsilon}(t)$. Then for $k = 1,2, \ldots$
\begin{itemize}
\item[(i)] $\sup_{t \in [0,T]} \E\abs{\xi^{\varepsilon}(t)}^{2k} = O(\varepsilon^k)$,
\item[(ii)] $\sup_{t \in [0,T]} \E\abs{y^{\varepsilon}(t)}^{2k}= O(\varepsilon^k)$,
\item[(iii)] $\sup_{t \in [0,T]} \E\abs{z^{\varepsilon}(t)}^{2k}= O(\varepsilon^{2k})$,
\item[(iv)] $\sup_{t \in [0,T]} \E\abs{\eta^{\varepsilon}(t)}^{2k} = O(\varepsilon^{2k})$,
\item[(v)] $\sup_{t \in [0,T]} \E\abs{\zeta^{\varepsilon}(t)}^{2k} = o(\varepsilon^{2k})$.

\end{itemize}
\end{prop}
\begin{proof}
The idea of the proof is to reduce every equation to a linear SDE with appropriate coefficients and to use lemma \eqref{l.eq.lin}.\\
\newline
\textbf{(i)} For $\xi^{\varepsilon}(t)$ we have
\begin{equation*}
\begin{split}
d\xi^{\varepsilon}(t) &= \left[ b(t,x^{\varepsilon}(t),u^{\varepsilon}(t)) - b(t,\bar{x}(t),\bar{u}(t))\right]dt \\
&+ \sum_{j=1}^d \left[ \sigma^j(t,x^{\varepsilon}(t),u^{\varepsilon}(t)) - \sigma^j(t,\bar{x}(t),\bar{u}(t)) \right]dW^j(t) \\
&= \left[ b(t,x^{\varepsilon}(t),u^{\varepsilon}(t)) - b(t,\bar{x}(t),u^{\varepsilon}(t)) + \delta b(t)\chi_{E_{\varepsilon}}(t) \right]dt \\
&+ \sum_{j=1}^d\left[ \sigma^j(t,x^{\varepsilon}(t),u^{\varepsilon}(t)) - \sigma^j(t,\bar{x}(t),u^{\varepsilon}(t)) + \delta\sigma^j(t)\chi_{E_{\varepsilon}}(t) \right]dW^j(t) \\
&= \int_0^1 D_xb(t,\bar{x}(t) + \theta\xi^{\varepsilon}(t),u^{\varepsilon}(t))d\theta\xi^{\varepsilon}(t) dt + \delta b(t)\chi_{E_{\varepsilon}}(t)dt \\
&+ \sum_{j=1}^d \int_0^1 D_x\sigma^j(t,\bar{x}(t) + \theta\xi^{\varepsilon}(t),u^{\varepsilon}(t))d\theta\xi^{\varepsilon}(t) dW^j(t) + \delta\sigma^j(t)dW^j(t) \\
&= \bigl[ G_b(t)\xi^{\varepsilon}(t) + \delta b(t)\chi_{E_{\varepsilon}}(t)\bigr]dt + \sum_{j=1}^d \bigl[ G^j_{\sigma}(t)\xi^{\varepsilon}(t) + \delta \sigma^j(t)\chi_{E_{\varepsilon}}(t)\bigr]dW^j(t)
\end{split}
\end{equation*}
where 
\[ G_b(t):= \int_0^1 D_xb(t,\bar{x}(t) + \theta\xi^{\varepsilon}(t),u^{\varepsilon}(t))d\theta, \]
\[ G_{\sigma}^j(t):= \int_0^1 D_x\sigma^j(t,\bar{x}(t) + \theta\xi^{\varepsilon}(t),u^{\varepsilon}(t))d\theta. \]
Next we apply Lemma \ref{l.eq.lin} noting
\[ \braket{G_b(t)\xi^{\varepsilon}(t),\xi^{\varepsilon}(t)} = \braket{\int_0^1 D_xb(t,\bar{x} + \theta\xi^{\varepsilon},u^{\varepsilon})\xi^{\varepsilon}(t) d\theta,\xi^{\varepsilon}(t)} \leq \alpha\abs{\xi^{\varepsilon}(t)}^2, \]
i.e. condition 1 of Lemma \ref{l.eq.lin} is verified.
Then we obtain
\begin{equation}\label{eq.deltab}
\begin{split}
\sup_{t\in [0,T]}\E\abs{\xi^{\varepsilon}(t)}^{2k} &\leq K\Bigl[ \int_0^T \abs{\delta b(t)\chi_{E_{\varepsilon}}(t)}_{L^{2k}(\Omega)}dt\Bigr]^{2k} + K\sum_{j=1}^d \Bigl[ \int_0^T \abs{\delta \sigma^{j}(t)\chi_{E_{\varepsilon}}(t)}_{L^{2k}(\Omega)}dt\Bigr]^{k}\\
&\leq K \Bigl[ \int_{E_{\varepsilon}} \abs{b(t,\bar{x}(t),u^{\varepsilon}(t)) - b(t,\bar{x}(t),\bar{u}(t))}_{L^{2k}(\Omega)}dt\Bigr]^{2k}  + K\varepsilon^{k}\\
&\leq K[\varepsilon^{2k} + \varepsilon^{k}] \leq K\varepsilon^{k}.
\end{split}
\end{equation}
thanks to the polynomial growth and \eqref{Lp_estimate}.\\
\newline
\textbf{(ii)}
Using the dissipativity of $D_xb$ and lemma \eqref{l.eq.lin} the estimate for $y^{\varepsilon}$ follows in the same way.\\
\newline
\textbf{(iii)} For $z^{\varepsilon}$ we have, proceeding as before,
\begin{equation*}
\begin{split}
\sup_{t\in [0,T]}\E\abs{z^{\varepsilon}(t)}^{2k} &\leq K\Bigl[ \int_0^T \abs{\delta b(t)\chi_{E_{\varepsilon}}(t) + \dfrac{1}{2}D^2_xb(t)y^{\varepsilon}(t)^2}_{L^{2k}(\Omega)}dt\Bigr]^{2k} \\
&+ K \sum_{j=1}^d \Bigl[ \int_0^T \Bigl(\E\abs{\delta D_x\sigma^j(t)\chi_{E_{\varepsilon}}(t)y^{\varepsilon}(t) + \dfrac{1}{2}D^2_x\sigma^j(t)y^{\varepsilon}(t)^2}^{2k}\Bigr)^{\frac{1}{k}}dt\Bigr]^{k}\\
&\leq K\Bigl[ \int_0^T \Bigl(\chi_{E_{\varepsilon}}(t) + K\varepsilon \Bigr)dt\Bigr]^{2k} + K \Bigl[ \int_0^T \Bigl(\chi_{E_{\varepsilon}}(t)\varepsilon + \varepsilon^2\Bigr)dt\Bigr]^{k} \\
&\leq K\varepsilon^{2k},
\end{split}
\end{equation*}

where we used the H\"{o}lder inequality, the estimate obtained in (ii) for $y^{\varepsilon}$, and the following
\begin{equation}\label{eq.general_estimate}
\bigl(\E\abs{D_x^2b(t)}^{4k}\bigr)^{\frac{1}{4k}} \leq K(1+ \E\abs{\bar{x}(t)}^{4kh})^{\frac{1}{4k}} \leq K(1 + \abs{x_0}^h) \leq K.
\end{equation}
\newline
\noindent \textbf{(iv)} Using the result obtained for $\xi^{\varepsilon}$ and $y^{\varepsilon}$ we can write
\begin{equation*}
\begin{split}
d\eta^{\varepsilon}(t) &= \left[ b(t,x^{\varepsilon}(t),u^{\varepsilon}(t)) - b(t,\bar{x}(t),u^{\varepsilon}(t)) + \delta b(t)\chi_{E_{\varepsilon}}(t) - D_xb(t)y^{\varepsilon}(t) \right]dt \\
&+ \sum_{j=1}^d \left[ \sigma^j(t,x^{\varepsilon}(t),u^{\varepsilon}(t)) - \sigma^j(t,\bar{x}(t),u^{\varepsilon}(t)) - D_x\sigma^j(t) y^{\varepsilon}(t) \right]dW^j(t) \\
&= D_x b(t)\eta^{\varepsilon}(t)dt + \left[ \int_0^1(Db_x(t,\bar{x}(t) + \theta\xi^{\varepsilon}(t),u^{\varepsilon}(t)) - D_xb(t)) d\theta\right]\xi^{\varepsilon}(t) dt\\
&+ \delta b(t)\chi_{E_{\varepsilon}}(t)dt + \sum_{j=1}^dD_x\sigma^j(t) \eta^{\varepsilon}(t) dW^j(t)\\
&+ \sum_{j=1}^d\left[ \int_0^1 (D_x\sigma^j(t,\bar{x}(t) + \theta\xi^{\varepsilon}(t),u^{\varepsilon}(t))  - D_x\sigma^j(t)) d\theta \right]\xi^{\varepsilon}(t) dW^j(t) \\
&= D_xb(t)\eta^{\varepsilon}(t) dt + \delta b(t)\chi_{E_{\varepsilon}}(t)dt + \left[G_b(t) - D_xb(t)\right]\xi^{\varepsilon}(t) dt\\
&+ \sum_{j=1}^d D_x\sigma^j(t) \eta^{\varepsilon}(t)dW^j(t) + \sum_{j=1}^d \bigl[(G_{\sigma}^j(t) - D_x^j\sigma(t))\xi^{\varepsilon}(t) \bigr]dW^j(t)\\
&= D_xb(t)\eta^{\varepsilon}(t) dt + \alpha^{\varepsilon}(t)dt + \sum_{j=1}^d \left[ D_x\sigma^j \eta^{\varepsilon}(t) + \beta^{j,\varepsilon}(t)\right] dW^j(t),
\end{split}
\end{equation*}
where we have defined 
\begin{equation*}
\begin{split}
\alpha^{\varepsilon}(t)&:= \delta b(t)\chi_{E_{\varepsilon}}(t) + \left[G_b(t) - D_xb(t)\right]\xi^{\varepsilon}(t);\\
\beta^{j,\varepsilon}(t)&:= (G_{\sigma}^j(t) - D_x^j\sigma(t))\xi^{\varepsilon}(t).
\end{split}
\end{equation*}
We begin studying $\alpha^{\varepsilon}(\cdot)$ as follows
\begin{equation*}
\begin{split}
\int_0^T \abs{\alpha^{\varepsilon}(t)}_{L^{2k}(\Omega)}dt &\leq K \int_0^T \Bigl[ \bigl( \E \abs{\delta b(t)\chi_{E_{\varepsilon}}(t)}^{2k} \bigr)^{\frac{1}{2k}} + \bigl( \E\abs{\bigl[ G_b(t) - D_xb(t) \bigr]\xi^{\varepsilon}(t)}^{2k} \bigr)^{\frac{1}{2k}} \Bigr] dt\\
&\leq K \int_{E_{\varepsilon}} \abs{\delta b(t)}_{L^{2k}(\Omega)}dt + K\varepsilon^{1/2} \int_0^T \abs{G_b(t) - D_xb(t)}_{L^{4k}(\Omega)}dt \\
\end{split}
\end{equation*}
where we used the same estimate as in \eqref{eq.deltab} for the first part and for the second term we applied the H\"{older} inequality. Now we want to estimate the last term. We have
\begin{equation*}
\begin{split}
G_b(t) - D_xb(t) &= \int_0^1\bigl[ D_xb(t,\bar{x}(t) + \theta\xi^{\varepsilon}(t),u^{\varepsilon}(t)) - D_xb(t)\bigr] d\theta \\
&= \int_0^1\bigl[D_xb(t,\bar{x}(t) + \theta\xi^{\varepsilon}(t),u^{\varepsilon}(t)) - D_xb(t,\bar{x}(t) + \theta \xi^{\varepsilon}(t),\bar{u}(t))\bigr] d\theta \\
&+ \int_0^1\bigl[ D_xb(t,\bar{x}(t) + \theta\xi^{\varepsilon}(t),\bar{u}(t)) - D_xb(t)\bigr] d\theta.
\end{split}
\end{equation*}
Here the second term has the same control process in both the integrands. Hence, using a Taylor expansion and the H\"{o}lder inequality, we get
\begin{equation}
\begin{split}
&\int_0^T \abs{\int_0^1\bigl[ D_xb(t,\bar{x}(t) + \theta\xi^{\varepsilon}(t),\bar{u}(t)) - D_xb(t)\bigr] d\theta}_{L^{4k}(\Omega)}dt \\
&= \int_0^T \abs{\int_0^1 D_x^2b(t,\tilde{x},\bar{u}(t))\theta\xi^{\varepsilon}(t) d\theta}_{L^{4k}(\Omega)}dt \\
&\leq \int_0^T \bigl(\E\abs{D_x^2b(t,\tilde{x},\bar{u}(t))}^{8k}\bigr)^{\frac{1}{8k}}\bigl(\E\abs{\xi^{\varepsilon}(t)}^{8k}\bigr)^{\frac{1}{8k}}dt \leq K\varepsilon^{1/2},
\end{split}
\end{equation}
where the last inequality follows from point $(i)$, the polynomial growth of $D_x^2b$  and \eqref{Lp_estimate}. Then we have

\begin{equation*}
\begin{split}
\int_0^T \abs{G_b(t) - D_xb(t)}_{L^{4k}(\Omega)}dt &\leq \int_{E_{\varepsilon}} \abs{\int_0^1\bigl[ D_xb(t,\bar{x}(t) + \theta\xi^{\varepsilon}(t),u^{\varepsilon}(t)) \\
&- D_xb(t,\bar{x}(t) + \theta \xi^{\varepsilon}(t),\bar{u}(t))\bigr] d\theta}_{L^{4k}(\Omega)}dt\\
&+ \int_0^T \abs{\int_0^1\bigl[ D_xb(t,\bar{x}(t) + \theta\xi^{\varepsilon}(t),\bar{u}(t)) - D_xb(t)\bigr] d\theta}_{L^{4k}(\Omega)}dt \\
&\leq K(\varepsilon + \varepsilon^{1/2}).
\end{split}
\end{equation*}
And we can conclude, in fact 
\begin{equation*}
\begin{split}
\int_0^T \abs{\alpha^{\varepsilon}(t)}_{L^{2k}(\Omega)}dt &\leq K \int_{E_{\varepsilon}} \abs{\delta b(t)}_{L^{2k}(\Omega)}dt + K\varepsilon^{1/2} \int_0^T \abs{G_b(t) - D_xb(t)}_{L^{4k}(\Omega)}dt \\
&\leq K \varepsilon. 
\end{split}
\end{equation*}
Regarding the estimate of $\beta^{\varepsilon}(t)$ we can proceed  
in the same way in order to obtain 
\[ \int_0^T \bigl[ \E\abs{G_{\sigma}^j(t) - D_x^j\sigma(t)}^{4k} \bigr]^{\frac{1}{2k}}dt \leq K\varepsilon. \]
From Lemma \ref{l.eq.lin} we have 
\begin{equation}
\begin{split}
\sup_t \E\abs{\eta^{\varepsilon}(t)}^{2k} &\leq K\left( \int_0^T \graf{\E\abs{\alpha(s)}^{2k}}^{\frac{1}{2k}}\right)^{2k} + K \sum_{j=1}^m \left( \int_0^T \graf{\E\abs{\beta^j(s)}^{2k}}^{\frac{1}{2k}}ds \right)^k \\
&\leq K[\varepsilon^{2k} + \varepsilon^{2k}] = O(\varepsilon^{2k}).
\end{split}
\end{equation}
\newline
\noindent \textbf{(v)} Also in this case we aim to use Lemma \ref{l.eq.lin}, combined with Lemma \ref{l.expansion}, in order to get the required estimate. If we write $d\zeta^{\varepsilon}(t) = d(\eta^{\varepsilon}(t) - \xi^{\varepsilon}(t))$ then the corresponding stochastic differential equation has the form
\[\begin{cases}
d\zeta^{\varepsilon}(t) = \left( D_xb(t)\zeta^{\varepsilon}(t) + \alpha^{\varepsilon}(t) \right)dt + \sum_{j=1}^d\left( D_x\sigma^j(t)\zeta^{\varepsilon}(t) + \beta^{j,\varepsilon}(t) \right)dW^j(t)\\
\zeta^{\varepsilon}(0) = 0;
\end{cases}\]
where
\begin{equation*}
\begin{split}
\alpha^{\varepsilon}(t) &:= \delta D_xb(t)\chi_{E_{\varepsilon}}(t)\xi^{\varepsilon}(t) + \dfrac{1}{2}\bigl[ \widetilde{G}_{b}(t) - D_x^2b(t,\bar{x}(t),u^{\varepsilon}(t)) \bigr]\xi^{\varepsilon}(t)^2 \\
&+ \dfrac{1}{2}\delta D_x^2b(t)\chi_{E_{\varepsilon}}(t)\xi^{\varepsilon}(t)^2 + \dfrac{1}{2}D_x^2b(t)[ \xi^{\varepsilon}(t)^2 - y^{\varepsilon}(t)^2 ],\\
&\quad \\
\beta^{\varepsilon}(t) &:= \delta D_x\sigma(t)\chi_{E_{\varepsilon}}(t)\eta^{\varepsilon}(t) + \dfrac{1}{2}\bigl[ \widetilde{G}_{\sigma}(t) - D_x^2\sigma(t,\bar{x}(t),u^{\varepsilon}(t)) \bigr]\xi^{\varepsilon}(t)^2 \\
&+ \dfrac{1}{2}\delta D_x^2\sigma(t)\chi_{E_{\varepsilon}}(t)\xi^{\varepsilon}(t)^2 + \dfrac{1}{2}D_x^2\sigma(t)[ \xi^{\varepsilon}(t)^2 - y^{\varepsilon}(t)^2 ],
\end{split}
\end{equation*}
and
\[\begin{cases}
\widetilde{G}_{b}(t) := 2 \int_0^1 \theta D_x^2b(t,\theta\bar{x}(t) + (1-\theta)x^{\varepsilon}(t),u^{\varepsilon}(t))d\theta,\\
\widetilde{G}_{\sigma}(t) := 2 \int_0^1 \theta D_x^2\sigma(t,\theta\bar{x}(t) + (1-\theta)x^{\varepsilon}(t),u^{\varepsilon}(t))d\theta.
\end{cases}\]
Indeed, using the equality in Lemma \ref{l.expansion}, for the drift part we have 
\begin{equation*}
\begin{split}
&b(t,x^{\varepsilon}(t),u^{\varepsilon}(t)) - b(t,\bar{x}(t),u^{\varepsilon}(t)) - D_xb(t)[y^{\varepsilon}(t) + z^{\varepsilon}(t)] - \dfrac{1}{2}D_x^2b(t)y^{\varepsilon}(t)^2 \\
&= D_xb(t,\bar{x},u^{\varepsilon}(t))\xi^{\varepsilon}(t) + \dfrac{1}{2}\widetilde{G}_b(t)\xi^{\varepsilon}(t)^2 - D_xb(t)[y^{\varepsilon}(t) + z^{\varepsilon}(t)] - \dfrac{1}{2}D_x^2b(t)y^{\varepsilon}(t)^2 \\
&= D_xb(t)\zeta^{\varepsilon}(t) + \alpha^{\varepsilon}(t).
\end{split}
\end{equation*}
For the diffusion term we can proceed in the same way
\begin{equation*}
\begin{split}
&\sigma^j(t,x^{\varepsilon}(t),u^{\varepsilon}(t)) - \sigma^j(t,\bar{x}(t),u^{\varepsilon}(t)) - D_x\sigma^j(t)[y^{\varepsilon}(t) + z^{\varepsilon}(t)] \\
&- \dfrac{1}{2}D_x^2\sigma^j(t)y^{\varepsilon}(t)^2  - \delta D_x\sigma^j(t)\chi_{E_{\varepsilon}}(t)y^{\varepsilon}(t) \\
&= D_x\sigma^j(t,\bar{x},u^{\varepsilon}(t))\xi^{\varepsilon}(t) + \dfrac{1}{2}\widetilde{G}_\sigma^j(t)\xi^{\varepsilon}(t)^2 - D_x\sigma^j(t)[y^{\varepsilon}(t) + z^{\varepsilon}(t)] \\
&- \dfrac{1}{2}D_x^2\sigma^j(t)y^{\varepsilon}(t)^2  - \delta D_x\sigma^j(t)\chi_{E_{\varepsilon}}(t)y^{\varepsilon}(t)\\
&= D_x\sigma^j(t)\zeta^{\varepsilon}(t) + \beta^{j,\varepsilon}(t).
\end{split}
\end{equation*}

Now we estimate $\alpha^{\varepsilon}(\cdot)$ as follows:
\begin{equation*}
\begin{split}
\int_0^T \abs{\alpha^{\varepsilon}(t)}_{L^{2k}(\Omega)}dt &\leq \int_0^T \Bigl[ \bigl( \E \abs{\delta D_xb(t)\chi_{E_{\varepsilon}}(t)\xi^{\varepsilon}(t)}^{2k} \bigr)^{\frac{1}{2k}} \\
&+ \dfrac{1}{2}\bigl( \E\abs{\bigl[ \widetilde{G}_b(t) - D_x^2b(t,\bar{x}(t),u^{\varepsilon}(t)) \bigr]\xi^{\varepsilon}(t)^2}^{2k} \bigr)^{\frac{1}{2k}} \\
&+ \dfrac{1}{2}\bigl( \E \abs{\delta D_x^2b(t)\chi_{E_{\varepsilon}}(t)\xi^{\varepsilon}(t)^2}^{2k}\bigr)^{\frac{1}{2k}} \\
&+ \dfrac{1}{2}\bigl( \E\abs{ D_x^2b(t)[ \xi^{\varepsilon}(t)^2 - y^{\varepsilon}(t)^2]}^{2k} \bigr)^{\frac{1}{2k}}\Bigr]dt \\
&\leq K \int_{E_{\varepsilon}}\bigl( \E\abs{\delta D_xb(t)}^{4k} \bigr)^{\frac{1}{4k}}\bigl( \E\abs{\xi^{\varepsilon}(t)}^{4k} \bigr)^{\frac{1}{4k}}dt \\
&+ K\int_0^T\bigl( \E\abs{\widetilde{G}_b(t) - D_x^2b(t,\bar{x}(t),u^{\varepsilon}(t))}^{4k} \bigr)^{\frac{1}{4k}}\bigl( \E\abs{\xi^{\varepsilon}(t)}^{8k} \bigr)^{\frac{1}{4k}}dt\\
&+K\int_{E_{\varepsilon}}\bigl( \E\abs{\delta D_x^2b(t)}^{4k} \bigr)^{\frac{1}{4k}}\bigl( \E\abs{\xi^{\varepsilon}(t)}^{8k} \bigr)^{\frac{1}{4k}}dt \\
&+ K \int_0^T \bigl( \E\abs{D_x^2b(t)}^{4k} \bigr)^{\frac{1}{4k}}\bigl( \E\abs{\eta^{\varepsilon}(t)}^{8k} \bigr)^{\frac{1}{8k}}\bigl( \E\abs{\xi^{\varepsilon}(t) + y^{\varepsilon}(t)}^{8k} \bigr)^{\frac{1}{8k}}dt \\
\end{split}
\end{equation*}
Here we used the polynomial growth of $b, D_xb, D_x^2 b$ as well as the a priori estimate \eqref{Lp_estimate} of the solution of the state equation. Let us focus on the second term and in particular on $\E\abs{\widetilde{G}_b(t) - D_x^2b(t,\bar{x}(t),u^{\varepsilon}(t))}^{4k}$. We get
\begin{equation}
\begin{split}
&\widetilde{G}_{b}(t) - D_x^2b(t,\bar{x}(t),u^{\varepsilon}(t))= \\
&=2 \int_0^1 \theta D_x^2b(t,\theta\bar{x}(t) + (1-\theta)x^{\varepsilon}(t),u^{\varepsilon}(t))d\theta - D_x^2b(t,\bar{x}(t),u^{\varepsilon}(t))\\
&= 2 \int_0^1 \theta \Bigl[ D_x^2b(t,\theta\bar{x}(t) + (1-\theta)x^{\varepsilon}(t),u^{\varepsilon}(t)) - D_x^2b(t,\bar{x}(t),u^{\varepsilon}(t))\Bigr] d\theta
\end{split}
\end{equation}

Now we want to show that the quantity above tends to zero as $\varepsilon \rightarrow 0$. Arguing by contradiction, we suppose that there exists a sequence $\varepsilon_n\rightarrow 0$ such that

\[ \int_0^T \abs{2 \int_0^1 \theta \Bigl[ D_x^2b(t,\theta\bar{x}(t) + (1-\theta)x^{\varepsilon_n}(t),u^{\varepsilon_n}(t)) - D_x^2b(t,\bar{x}(t),u^{\varepsilon_n}(t))\Bigr] d\theta}_{L^{4k}(\Omega)}dt \geq \delta > 0, \]
but from point $(i)$ we have that $\sup_t\abs{\xi^{\varepsilon_n}(t)}_{L^{4k}(\Omega)} \rightarrow 0$, hence there is a subsequence $\varepsilon_{n_k}$ such that $\xi^{\varepsilon}_{n_k} \rightarrow 0$, that is $x^{\varepsilon}_{n_k} \rightarrow \bar{x}$ $d\mP \times dt$-a.s.. Now, using dominated convergence theorem (i.e. $D_x^2b$ has polynomial growth), thanks to the continuity of $D_x^2b$, we get
\[ \int_0^T \abs{2 \int_0^1 \theta \Bigl[ D_x^2b(t,\theta\bar{x}(t) + (1-\theta)x^{\varepsilon}_{n_k}(t),u^{\varepsilon}_{n_k}(t)) - D_x^2b(t,\bar{x}(t),u^{\varepsilon}_{n_k}(t))\Bigr] d\theta}_{L^{4k}(\Omega)}dt \rightarrow 0, \]
that is absurd. Finally we have
\begin{equation*}
\begin{split}
\int_0^T \abs{\alpha^{\varepsilon}(t)}_{L^{2k}(\Omega)}dt &\leq K \int_0^T \bigl( \E\abs{\widetilde{G}_b(t) - D_x^2b(t,\bar{x}(t),u^{\varepsilon}(t))}^{4k} \bigr)^{\frac{1}{4k}}\bigl( \E\abs{\xi^{\varepsilon}(t)}^{8k} \bigr)^{\frac{1}{4k}}dt\\
&+K \int_{E_{\varepsilon}}\bigl( \E\abs{\xi^{\varepsilon}(t)}^{4k} \bigr)^{\frac{1}{4k}}dt + K\int_{E_{\varepsilon}}\bigl( \E\abs{\xi^{\varepsilon}(t)}^{8k} \bigr)^{\frac{1}{4k}}dt\\
&+ K \int_0^T \bigl( \E\abs{\eta^{\varepsilon}(t)}^{8k} \bigr)^{\frac{1}{8k}}\bigl( \E\abs{\xi^{\varepsilon}(t) + y^{\varepsilon}(t)}^{8k} \bigr)^{\frac{1}{8k}}dt \\
&\leq o(\varepsilon) + K[ \varepsilon^{3/2} + \varepsilon^2 + \varepsilon^{3/2} ] = o(\varepsilon).
\end{split}
\end{equation*}
For $\beta^{\varepsilon}(t)$ we use the boundedness of the derivative of $\sigma(\cdot)$ and proceeding in the same way we obtain
\begin{equation*}
\int_0^T \abs{\beta^{j,\varepsilon}(t)}_{L^{2k}(\Omega)}dt = o(\varepsilon^2)
\end{equation*}
Then, thanks to Lemma \ref{l.eq.lin} the desired result follows.
\end{proof}

Focusing on the cost functional, now we deduce a Taylor expansion of the cost with respect to the spike variation of the control process in order to use some duality argument.

\begin{prop}\label{p.cost_expansion}
Under assumptions (H1)-(H5) we have the following expansion of the cost functional
\begin{equation*}
\begin{split}
J(u^{\varepsilon}(\cdot)) - J(\bar{u}(\cdot)) &= \E \int_0^T \Bigl[ \braket{D_xf(t), y^{\varepsilon}(t) + z^{\varepsilon}(t)} + \dfrac{1}{2}\braket{D_x^2f(t)y^{\varepsilon}(t), y^{\varepsilon}(t)} + \delta f(t)\chi_{E_{\varepsilon}}\Bigr]dt \\
&+ \E\braket{h_x(\bar{x}(T)),y^{\varepsilon}(T) + z^{\varepsilon}(T)} + \dfrac{1}{2}\E\braket{D_x^2h(\bar{x}(T))y^{\varepsilon}(T), y^{\varepsilon}(T)} + o(\varepsilon).
\end{split}
\end{equation*}
where $D_xf(t):= D_xf(t,\bar{x}(t),\bar{u}(t))$ and $D^2_xf(t):= D^2_xf(t,\bar{x}(t),\bar{u}(t))$.
\end{prop}

\begin{proof}
Using Lemma \ref{l.expansion}, Proposition \ref{p.expansion} and the polynomial growth of $f$ and $h$, this is a straightforward calculation (see for example \cite{YZ}, Theorem 4.4, page 133).
\end{proof}

\section{Proof of Theorem \ref{t.SMP}} 

In the preceding section we studied how the optimal trajectory varies after a small perturbation of the control process. The goal was to have an expansion of the cost and produce a preliminary necessary condition for a given optimal pair. In particular what we obtained is the following
\begin{equation}\label{eq.preliminary_condition}
\begin{split}
0 &\leq \E \int_0^T \Bigl[ \braket{D_xf(t), y^{\varepsilon}(t) + z^{\varepsilon}(t)} + \dfrac{1}{2}\braket{D_x^2f(t)y^{\varepsilon}(t), y^{\varepsilon}(t)} + \delta f(t)\chi_{E_{\varepsilon}}\Bigr]dt \\
&+ \E\braket{h_x(\bar{x}(T)),y^{\varepsilon}(T) + z^{\varepsilon}(T)} + \dfrac{1}{2}\E\braket{D_x^2h(\bar{x}(T))y^{\varepsilon}(T), y^{\varepsilon}(T)} + o(\varepsilon).
\end{split}
\end{equation}
Now we are in position to conclude the proof of the SMP.

\begin{proof}[Proof of Theorem \ref{t.SMP}]
Using the results obtained in the previous section, we are in position to conclude the proof of the theorem as in \cite{YZ} for the classical setting. For the sake of completeness we give an outline of it. Using It\^{o}'s formula to compute $d\braket{p(t),y^{\varepsilon}(t)}$ and $d\braket{p(t),z^{\varepsilon}(t)}$ it is easy to derive the following equalities

\begin{equation}\label{eq.prodotto_py}
\E\braket{p(T),y^{\varepsilon}(T)} = \E\int_0^T \bigl[ \braket{D_xf(t),y^{\varepsilon}(t)} + \opn{Tr}(q(t)^T \delta \sigma(t))\chi_{E_{\varepsilon}}(t) \bigr]dt;
\end{equation}

\begin{equation}\label{eq.prodotto_pz}
\begin{split}
\E\braket{p(T),z^{\varepsilon}(T)} &= \E\int_0^T \bigl[ \braket{D_xf(t),z^{\varepsilon}(t)} + \dfrac{1}{2}\big( \braket{p(t),D_x^2b(t)y^{\varepsilon}(t)^2}\\
&+ \sum_{j=1}^d\braket{q_j(t), D_x^2\sigma^j(t)y^{\varepsilon}(t)^2}\bigr)\bigr]dt \\
&+ \E \int_0^T \bigl[ \braket{p(t),\delta b(t)}  + \sum_{j=1}^d\braket{q_j(t), \delta D_x\sigma^j(t)y^{\varepsilon}(t)} \bigr]\chi_{E_{\varepsilon}}(t)dt
\end{split}
\end{equation}
hence, recalling that $p(T) = -D_xh(\bar{x}(T))$ and adding \eqref{eq.prodotto_py} and \eqref{eq.prodotto_pz}, we get
\begin{equation*}
\begin{split}
-\E \braket{D_x h(\bar{x}(T)),y^{\varepsilon}(T) + z^{\varepsilon}(T)} 
&= \E \int_0^T  \bigl[\braket{D_xf(t),y^{\varepsilon}(t) + z^{\varepsilon}(t)} + \dfrac{1}{2}\braket{p(t),D_x^2b(t)y^{\varepsilon}(t)^2}\bigr]dt\\
&+ \dfrac{1}{2}\sum_{j=1}^d\E \int_0^T  \braket{q_j(t),D_x^2\sigma^j(t)y^{\varepsilon}(t)^2}dt \\
&+ \bigl[ \braket{p(t),\delta b(t)} + \opn{Tr}(q(t)^T\delta \sigma(t))\chi_{E_{\varepsilon}}(t)\bigr]dt + o(\varepsilon).
\end{split}
\end{equation*}
Thanks to the optimality of $\bar{u}(t)$, substituting the above term in the expression of the cost given in Proposition \ref{p.cost_expansion}, we have
\begin{equation*}
\begin{split}
0 &\geq J(\bar{u}) - J(u^{\varepsilon}) \\
&= -\dfrac{1}{2}\E\braket{D_x^2h(T)y^{\varepsilon}(T), y^{\varepsilon}(T)} + \dfrac{1}{2}\E\int_0^T\bigl[ - \braket{D_x^2f(t)y^{\varepsilon}(t), y^{\varepsilon}(t)}  + \braket{p(t),D_x^2b(t)y^{\varepsilon}(t)^2}\bigr]dt\\
&+\dfrac{1}{2}\sum_{j=1}^d\E \int_0^T\braket{q_j(t),D_x^2\sigma^j(t)y^{\varepsilon}(t)^2}dt + \E\int_0^T \bigl[ -\delta f(t) + \braket{p(t),\delta b(t)} \bigr]\chi_{E_{\varepsilon}}(t)dt \\
&+ \sum_{j=1}^d\E\int_0^T \braket{q_j(t),\delta \sigma^j(t)}\chi_{E_{\varepsilon}}(t) + o(\varepsilon).
\end{split}
\end{equation*}
Introducing another matrix valued process $Y^{\varepsilon}(t):= y^{\varepsilon}(t)y^{\varepsilon}(t)^T$ in order to get rid of the second order terms in $y^{\varepsilon}(t)$, we get

\begin{equation*}
\begin{split}
0 &\geq J(\bar{u}) - J(u^{\varepsilon}) \\
&= \dfrac{1}{2}\E\opn{Tr}\bigl( P(T)Y^{\varepsilon}(T) \bigr) + \E\int_0^T \bigl[ \dfrac{1}{2}\opn{Tr}(D_x^2H(t)Y^{\varepsilon}(t)) + \delta H(t)\chi_{E_{\varepsilon}}(t) \bigr]dt + o(\varepsilon),
\end{split}
\end{equation*}
where $D_x^2H(t) := D_x^2H(t,\bar{x}(t),\bar{u}(t),p(t),q(t))$. Then if we use the duality relation of \cite{YZ}, Lemma 4.6, page 137, in the form 
\begin{equation}
\E \opn{Tr}(P(T)Y(T)) = \E \int_0^T \opn{Tr}\bigl[\delta \sigma(t)^TP(t)\delta\sigma(t)\chi_{E_{\varepsilon}}(t) - D_x^2H(t)Y^{\varepsilon}(t) \bigr]dt + o(\varepsilon)
\end{equation}
we can eventually get the following
\begin{equation}\label{eq.quasi_result}
o(\varepsilon) \geq \E \int_0^T \bigl[ \delta H(t) + \dfrac{1}{2}\opn{Tr}(\delta\sigma(t)^TP(t)\delta\sigma(t))\bigr]\chi_{E_{\varepsilon}}(t)dt
\end{equation}

Finally, from the above expression \eqref{eq.quasi_result} we obtain that 
\begin{equation*}
H(t,\bar{x}(t),u(t),p(t),q(t)) - H(t,\bar{x}(t),\bar{u}(t),p(t),q(t)) + \dfrac{1}{2}\opn{Tr}(\delta\sigma(t)^TP(t)\delta\sigma(t)) \leq 0,
\end{equation*}
$\forall u \in U$, a.e. $t \in [0,T]$, $\mP$-a.s.. If we rewrite it in term of $\mathcal{H}$ we get the result.
\end{proof}

\section{The Convex Case}
As we mentioned at the beginning of Section 2, here we are going to discuss the case where controls take values in a closed convex subset $U$ of $\R^n$. In the following we are going to obtain a version of the SMP using the convexity of $U$ and later to derive a sufficient condition of optimality.\\ \\
\remark In this section the maps $D_x^2b, D_x^2\sigma, D_x^2f, D_x^2h$ are no longer used. So, from now on, when we refer to hypothesis (H2)-(H5) we will assume that all the maps involved are only $C^1$ with respect to $x$, in contrast with the previous sections. It is even worth noting that we still have a polynomial growth condition on the first derivative.

\subsection{Necessary conditions} 
The convexity assumption allows us to use a perturbation argument instead of a spike variation technique, avoiding the introduction of the second adjoint equation. On the other hand, in order to treat this case we have to make another assumption: 
\newline \\
HYPOTHESIS (H6) \; The control domain $U$ is a convex subset of $\R^n$. If $\varphi = b,\sigma, f$, the maps $u \mapsto \varphi(t,x,u)$ are $C^1(U)$ and their derivatives satisfy a polynomial growth such as
\[\abs{D_u\varphi(t,x,u)} \leq C(1+ \abs{x}^k), \text{ for some } k \in \mathbb{N}.\]

If $\bar{u}(\cdot)$ is an optimal control we will consider $\bar{u}(\cdot) + \theta(u(\cdot)-\bar{u}(\cdot))$, where $u(\cdot)$ is admissible and we set $x_{\theta}(t)$ the trajectory corresponding to the perturbed control. 
The optimality of $\bar{u}(\cdot)$ guaranties that 
\[ J\bigl(\bar{u}(\cdot) + \theta(u(\cdot)-\bar{u}(\cdot))\bigr) \geq J(\bar{u}(\cdot)). \]
We have to prove that $J(\cdot)$, considered as a functional on $L^1_{\mF}(0,T)$, is G\^ateaux differentiable. Then we will write 
\[ \braket{J'(\bar{u}),u(\cdot)-\bar{u}(\cdot)} \geq 0, \qquad \forall \; u(\cdot) \in \mathcal{U}[0,T], \]
and we will deduce a form of the SMP.\\
If we define a new process $y(t)$ as a solution of the stochastic differential equation
\begin{equation}\label{eq.yu}
\begin{sistema}
dy(t) = \left[ D_xb(t,\bar{x}(t),\bar{u}(t))y(t) + D_ub(t,\bar{x}(t),\bar{u}(t))u(t)\right]dt \\
\qquad \quad + \left[ D_x\sigma(t,\bar{x}(t),\bar{u}(t))y(t) + D_u\sigma(t,\bar{x}(t),\bar{u}(t))u(t)\right]dW(t), \\
y(0) = 0,
\end{sistema}
\end{equation}
we can state the following

\begin{lemma}\label{l.gateaux}
The functional $J(\cdot)$ is G\^ateaux differentiable, moreover the derivative has the form
\begin{equation}\label{gateaux}
\dfrac{d}{d\theta}J(\bar{u}(\cdot)+ \theta u(\cdot))\big|_{\theta=0} = \E \left[ \braket{D_xh(T),y(T)} + \xi(T) \right]
\end{equation}
where $\xi$ is the solution to 
\begin{equation*}
\begin{sistema}
\dfrac{d\xi}{dt} = D_xf(t,\bar{x}(t),\bar{u}(t))y(t) + D_uf(t,\bar{x}(t),\bar{u}(t))u(t)\\
\xi(0) = 0,
\end{sistema}
\end{equation*}
\end{lemma}

\begin{proof}
We denote $x_{\theta}$ the trajectory corresponding to the perturbed control and set
\[ \tilde{x}_{\theta}(t) = \dfrac{x_{\theta}(t) - x(t)}{\theta} - y(t). \]
The idea of the proof is to show that $\abs{\tilde{x}_{\theta}(t)}_{L^2(\Omega)}^2 \rightarrow 0$ when $\theta \rightarrow 0$. In fact, this is crucial in order to show that
\begin{equation}
\dfrac{1}{\theta}\E \bigl[ h(x_{\theta}(T)) - h(x(T))\bigl] \longrightarrow \E\braket{D_xh(x(T)),y(T)}.
\end{equation}
We start by writing the equation for $\tilde{x}_{\theta}(t)$
\begin{equation}
\begin{split}
d\tilde{x}_{\theta}(t) &= \frac{1}{\theta}\bigl[ b(t,\bar{x}(t) + \theta y(t) + \theta \tilde{x}_{\theta}(t),\bar{u}(t)+\theta u(t))\\
&- b(t,\bar{x}(t),\bar{u}(t)) - \theta D_xb(t) y(t) - \theta D_ub(t) u(t) \bigr]dt \\
&+ \frac{1}{\theta}\bigl[ \sigma(t,\bar{x}(t) + \theta y(t) + \theta \tilde{x}_{\theta}(t),\bar{u}(t)+\theta u(t))\\ 
&- \sigma(t,\bar{x}(t),\bar{u}(t)) - \theta D_x\sigma(t)y(t) - \theta D_u\sigma(t)u(t) \bigr]dW(t)
\end{split}
\end{equation}
with $\tilde{x}_{\theta}(0) = 0 $ as initial condition. Then using the same technique as in the spike variation case we get the following equation 
\begin{equation*}
\begin{split}
d\tilde{x}_{\theta}(t) &= \int_0^1 D_xb\bigl(t,\bar{x}(t) + \lambda\theta( y(t) + \tilde{x}_{\theta}(t)),\bar{u}(t)+ \lambda\theta u(t)\bigr)\tilde{x}_{\theta}(t)\, d\lambda dt\\
&+ \int_0^1 D_x\sigma\bigl(t,\bar{x}(t) + \lambda\theta( y(t) + \tilde{x}_{\theta}(t)),\bar{u}(t)+ \lambda\theta u(t)\bigr)\tilde{x}_{\theta}(t)\, d\lambda dW(t)\\
&+ \int_0^1 \bigl[ D_xb\bigl(t,\bar{x}(t) + \lambda\theta( y(t) + \tilde{x}_{\theta}(t)),\bar{u}(t)+ \lambda\theta u(t)\bigr) - D_xb(t)\bigr]y(t)\, d\lambda dt\\
&+ \int_0^1 \bigl[ D_x\sigma\bigl(t,\bar{x}(t) + \lambda\theta( y(t) + \tilde{x}_{\theta}(t)),\bar{u}(t)+ \lambda\theta u(t)\bigr) - D_x\sigma(t)\bigr]y(t)\, d\lambda dW(t)\\
&+ \int_0^1 \bigl[ D_ub\bigl(t,\bar{x}(t) + \lambda\theta( y(t) + \tilde{x}_{\theta}(t)),\bar{u}(t)+ \lambda\theta u(t)\bigr) - D_ub(t)\bigr]u(t)\, d\lambda dt\\
&+ \int_0^1 \bigl[ D_u\sigma\bigl(t,\bar{x}(t) + \lambda\theta( y(t) + \tilde{x}_{\theta}(t)),\bar{u}(t)+ \lambda\theta u(t)\bigr) - D_u\sigma(t)\bigr]u(t)\, d\lambda dW(t).
\end{split}
\end{equation*}
Applying It\^{o} formula to the function $\tilde{x}_{\theta} \mapsto \abs{\tilde{x}_{\theta}}^2$ and taking the expectation we get
\begin{equation*}
\begin{split}
\E |\tilde{x}_{\theta}|^2 &\leq K\E \int^t_0 |\tilde{x}_{\theta}(s)|^{2} ds\\
&+ K\E \int_0^T |y(t)|^2 \int_0^1 |D_xb(t,\bar{x}(t) + \lambda\theta( y(t) + \tilde{x}_{\theta}(t)),\bar{u}(t)+ \lambda\theta u(t)) - D_xb(t)|^2 d\lambda dt \\
&+ K\E \int_0^T |u(t)|^2 \int_0^1 |D_ub(t,\bar{x}(t) + \lambda\theta( y(t) + \tilde{x}_{\theta}(t)),\bar{u}(t)+ \lambda\theta u(t)) - D_ub(t)|^2 d\lambda dt\\
&+ \E \int_0^T |y(t)|^2\int_0^1 \bigl[ D_x\sigma(t,\bar{x}(t) + \lambda\theta( y(t) + \tilde{x}_{\theta}(t)),\bar{u}(t)+ \lambda\theta u(t)) - D_x\sigma(t)\bigr]^2 d\lambda dt\\
&+ \E\int_0^T |u(t)|^2\int_0^1 \bigl[ D_u\sigma(t,\bar{x}(t) + \lambda\theta( y(t) + \tilde{x}_{\theta}(t)),\bar{u}(t)+ \lambda\theta u(t)) - D_u\sigma(t)\bigr]^2 d\lambda dt \\
&= K\E \int^t_0 |\tilde{x}_{\theta}(s)|^{2} ds + \rho_{\theta},
\end{split}
\end{equation*}
thanks to the polynomial growth of $D_xb, D_x\sigma, D_ub, D_u\sigma$ and the Young inequality. Now, let us estimate the second term of the right hand side of the above inequality. If $\theta \rightarrow 0$ then also
\[ \E\int_0^1 |D_xb\bigl(t,\bar{x}(t) + \lambda\theta( y(t) + \tilde{x}_{\theta}(t)),\bar{u}(t)+ \lambda\theta u(t)\bigr) - D_xb(t)|^2 d\lambda \longrightarrow 0, \]
due to the polynomial growth and the continuity of $D_xb$ with respect to $(x,u)$. For the remaining terms the same argument applies, so we can conclude that if $\theta \rightarrow 0$ also  $\rho_{\theta} \rightarrow 0$.
Finally applying Gronwall inequality we get
\[ \E |\tilde{x}_{\theta}(t)|^2 \leq K\rho_{\theta} \longrightarrow 0 \qquad \text{ if } \theta \rightarrow 0 \]
Then, in order to prove formula \eqref{gateaux} one has to compute the following
\begin{itemize}
\item[(i)] $\E\dfrac{1}{\theta} [h(x_{\theta}(T)) - h(x(T))] \longrightarrow \E\braket{D_xh(x(T)),y(T)}$
\item[(ii)] $\E \dfrac{1}{\theta} \int_0^T \bigr[ f(t,x_{\theta},\bar{u} + \theta u) - f(t,\bar{x},\bar{u}) \bigl]dt \longrightarrow \E\xi(T).$
\end{itemize}
but $(i)$ can be rewritten in the form
\begin{equation}
\begin{split}
&\E \int_0^1 D_xh(\bar{x}(T) + \lambda (x_{\theta}(T) - \bar{x}(T)))(\tilde{x_{\theta}}(T) + y(T))d\lambda \\
&\leq \int_0^1 \E \bigl(\abs{D_xh(\bar{x}(T) + \lambda(x_{\theta}(T) - \bar{x}(T)))}^2\bigr)^{\frac{1}{2}}\bigl(\E\abs{\tilde{x}_{\theta}(T)}^2\bigr)^{\frac{1}{2}}d\lambda \\
&+ \E \int_0^1 D_xh(\bar{x}(T) + \lambda (x_{\theta}(T) - \bar{x}(T)))y(T)d\lambda  
\end{split}
\end{equation}
where we used the H\"{o}lder inequality. Passing to the limit with $\theta \rightarrow 0$ we can conclude. Regarding $(ii)$, the result follows in a similar way. 
\end{proof}

\noindent Now we can state the maximum principle also in this particular case, where controls assume their values in a convex subset.

\begin{thm}
Suppose $(H2)-(H6)$ hold and let $(\bar{x}(\cdot),\bar{u}(\cdot))$ be an optimal pair for the control problem \eqref{PROBLEMA}. Then there exist $p, q_j \in L^2_{\mathcal{F}}(0,T;\R^n)$ which are a solution of the BSDE \eqref{eq.adjoint_first}, such that
\[ \braket{\dfrac{\partial H}{\partial u}(t,\bar{x}(t), \bar{u}(t),p(t)),u - \bar{u}(t)} \leq 0 \qquad \qquad d\mP \times dt \; q.c. , \; u \in U\]
where $H$ is the Hamiltonian \eqref{hamiltonian}. 
\end{thm} 

\begin{proof}
The existence  and uniqueness of a solution to the BSDE \eqref{eq.adjoint_first} is guaranteed, due to Theorem \ref{t.BSDE}. Moreover, thanks to \eqref{eq.yu} and Lemma \ref{l.gateaux} we are able to compute 
\begin{equation}
\E \left[ d\braket{p(t),y(t)}\right] = \E \bigl[ \braket{D_xf(t),y(t)} + \braket{p(t), D_ub(t)u(t)} + \sum_{j=1}^d\braket{q_j(t),D_u\sigma^j(t)u(t)} \bigr]dt
\end{equation}
and we know that
\begin{equation*}
\begin{split}
-\E &\braket{D_xh(\bar{x}(T)),y(T)} = \E \braket{p(T),y(T)} - \E \braket{p(0),y(0)} \\
&= \E \int_0^T  \bigl[\braket{D_xf(t),y(t)} + \braket{p(t), D_ub(t)u(t)} + \sum_{j=1}^d\braket{q_j(t),D_u\sigma^j(t)u(t)} \bigr]dt.
\end{split}
\end{equation*}
Hence, from Lemma \ref{l.gateaux}
\begin{equation*}
\begin{split}
0 &\leq \dfrac{d}{d\theta} J(\bar{u} + \theta u)|_{\theta =0} = \E \int_0^T \left[\braket{D_uf(t),u(t)} - \braket{p(t),D_ub(t)^T u(t)} - \braket{q(t),D_u\sigma(t) u(t)}\right]dt\\
&= \E \int_0^T \langle\dfrac{\partial}{\partial u} \bigl[f(t,\bar{x}(t),\bar{u}(t)) - p(t)\cdot b(t,\bar{x}(t),\bar{u}(t)) - \opn{Tr}(q(t)\sigma^T(t,\bar{x}(t),\bar{u}(t)))\bigr],u(t)\rangle dt \\
&= -\E \int_0^T \braket{\dfrac{\partial H}{\partial u}(t,\bar{x}(t),\bar{u}(t),p(t),q(t)) ,u(t)}dt,
\end{split}
\end{equation*}
for all  $u(\cdot) \in \mathcal{U}[0,T]$, and we have finished.
\end{proof}  

\begin{ex} 
{\rm Even in this setting we can consider drift terms of polynomial type as in Example \ref{ex1}.}
\end{ex}

\subsection{Sufficient conditions}

Here we want to remark that also in our framework it is possible to derive a sufficient condition of optimality of a pair $(\bar{x},\bar{u})$. In particular, unlike the previous paragraph it is not necessary to ask for the differentiability of coefficients with respect to the control. Indeed only a locally Lipschitz assumption is needed along with some simple properties of Clarke's generalized gradient.\\
\newline
HYPOTHESIS (H7) \; The control domain $U$ is a convex subset of $\R^n$. If $\phi = b, \sigma, f$, the maps $u \mapsto \phi(t,x,u)$ are locally Lipschitz in $u$ and their derivatives with respect to $x$, i.e. $D_x\phi(t,x,u)$, are continuous in $(x,u)$.

\begin{thm}
Let hypotheses (H2)-(H5) and (H7) hold. Let $(\bar{x}(\cdot)),\bar{u}(\cdot)$ be an admissible pair, $(p(\cdot),q(\cdot))$ and $(P(\cdot),Q(\cdot))$ be solutions to \eqref{eq.adjoint_first}, \eqref{eq.adjoint_second}, respectively. If the following assumptions hold
\begin{itemize}
\item[(i)] $h(\cdot)$ is a convex function;
\item[(ii)] the Hamiltonian $H(t,\cdot,\cdot,p(t),q(t))$ is concave for all $t \in [0,T]$, $\mP$-a.s.;
\item[(iii)] $\mathcal{H}(t,\bar{x}(t), \bar{u}(t)) = \max_{u\in U}\mathcal{H}(t,\bar{x}(t),u)$, a.e. $ t \in [0,T]$, $\mP$-a.s..
\end{itemize}  
Then $(\bar{x}(\cdot),\bar{u}(\cdot))$ is an optimal pair of the problem \eqref{PROBLEMA}. 
\end{thm}
\begin{proof}
The key fact of the proof (see \cite{YZ}, Lemma 5.1, page 138) is to show that 
\[ \partial_u H(t,\bar{x}(t),\bar{u}(t),p(t),q(t)) = \partial_u \mathcal{H}(t,\bar{x}(t),\bar{u}(t)) \]
where $\partial_uH$ is the Clarke's generalized gradient of the Hamiltonian. Then the proof proceed exactly as in \cite{YZ} (page 139-140) noting that the first adjoint equation \eqref{eq.adjoint_first} is well posed.
\end{proof}
 
\ack The author would like to thank Professor Marco Fuhrman for helpful discussions and suggestions related to this work.

\end{document}